\documentclass[10pt,reqno]{amsart}
\usepackage{geometry}
\usepackage{amsthm}
\usepackage{amsmath}
\usepackage{amssymb}
\usepackage{mathrsfs}
\usepackage{latexsym}
\usepackage{exscale}
\usepackage{geometry}
\usepackage{graphicx}
\usepackage[utf8]{inputenc}
\usepackage{verbatim}

\usepackage{fancyhdr}

\usepackage{color}
\definecolor{red}{rgb}{1,0.1,0.1}
\definecolor{blue}{rgb}{0.1,0.1,1}
\definecolor{vb}{RGB}{160,32,240}

\theoremstyle{plain}

\newtheorem{theorem}{Theorem}
\newtheorem{remark}{Remark}
\newtheorem{lemma}{Lemma}
\newtheorem{definition}{Definition}
\usepackage{ragged2e}

\newcommand{\D}{\mathrm{d}}
\newcommand{\dist}{\mathrm{dist}}

\begin{document}
\title[Existence of positive solutions for a fractional $p$-Laplacian problem]{Existence of positive solutions for a parameter fractional $p$-Laplacian problem with semipositone nonlinearity}

\author[E.~Lopera]{Emer Lopera}
\address{E.~Lopera \\ Universidad Nacional de Colombia, sede Manilzales,\\
	Manizales 170003, Colombia }
\email{edloperar@unal.edu.co}

\author[C.~L\'opez]{Camila López}
\address{C.~L\'opez \\ Universidad Nacional de Colombia, sede Manilzales,\\
	Manizales 170003, Colombia  }
\email{camlopezmor@unal.edu.co}

\author[R.~E.~Vidal]{Ra\'ul E. Vidal}
\address{R.~E.~Vidal \\ FaMAF \\ Universidad Nacional de C\'ordoba \\
	CIEM (CONICET) \\ 5000 C\'ordoba, Argentina}
\email{vidal@famaf.unc.edu.ar}

	
	

\begin{abstract}
	\justifying{In this paper we prove the existence of  at least one positive solution for the nonlocal semipositone problem
		\[
		\displaystyle \left\{\begin{array}{rcll}
			(-\Delta)_p^s(u) &=& \lambda f(u)  \qquad & \text{in} \ \ \Omega \\
			u  &=&   0   & \text{in} \ \ \mathbb{R}^N -\Omega ,        
		\end{array}\right.
		\]
		whenever $\lambda >0$ is a sufficiently small parameter. Here $\Omega \subseteq \mathbb{R}^N$ a bounded domain with $C^{1,1}$ boundary, $2\leqslant p <N$, $s\in (0,1)$ and $f$ superlineal and subcritical. We prove that if $\lambda>0$ is chosen sufficiently small the associated Energy Functional to the problem has a mountain pass structure and, therefore, it has a critical point $u_\lambda$, which is a weak solution. After that we manage to prove that  this solution is positive by using new regularity results up to the boundary and a Hopf's Lemma.}
\end{abstract}

	\maketitle

\section{Introduction}
\noindent
We are interested in the study of the existence of positive solutions to the problem
\begin{equation}\label{eq1}
	\displaystyle \left\{\begin{array}{rcll}
		(-\Delta)_p^s(u) &=& \lambda f(u)  \qquad & \text{in} \ \ \Omega \\
		u  &=&   0   & \text{in} \ \ \mathbb{R}^N -\Omega ,        
	\end{array}\right.
\end{equation}
\noindent
where $N>2$ is an integer, $\Omega \subseteq \mathbb{R}^N$ is a bounded domain with $C^{1,1}$ boundary, $s\in (0,1)$, $1< p$ and $sp<N$ and $\lambda >0$. Besides $f:\mathbb{R} \to \mathbb{R}$ is a continuous function and $(-\Delta )_p^s$ is the $s$-fractional $p$-Laplacian operator defined as \[(-\Delta )_p^s u (x)=2\lim_{\varepsilon \to 0^+} \int_{|x-y|>\varepsilon} \frac{|u(x)-u(y)|^{p-2}(u(x)-u(y))}{|x-y|^{N+sp}} \D y.\]
Let us denote by $p_s^*:=\frac{Np}{N-sp}$ the fractional critical Sobolev exponent. For any Lebesgue measurable set $U\subseteq \mathbb{R}^N$, $|U|$ will stand for the Lebesgue measure of $U$. 
In this work we will assume that there exist  $p-1 <q< \min \{ \frac{sp}{N}p_s^* , p_s^*-1 \} $, $A,B>0$ such that
\begin{equation}\label{eqn4}
	\begin{array}{lll}
		A(s^q-1) \leqslant &f(s)  \leqslant B(s^q+1)      &\text{for} \ \  s >0 \\
		&f(s)=0     &\text{for} \ \ s \leqslant -1 
	\end{array}.
\end{equation}
Let us define \[F(t):=\int_0^t f(s)\D s.\]
Therefore, there exist $A_1, C_1,B_1 >0$ such that
\begin{equation}\label{upper_bound_F}
	F(u)  \leqslant B_1(|u|^{q+1}+1)\qquad \text{for all} \ \  u \in \mathbb{R} 
\end{equation}
and 
\begin{equation}\label{lower_bound_F}
	A_1(u^{q+1}-C_1) \leqslant  F(u)   \qquad    \text{for all} \ \ u \geqslant 0 .
\end{equation}
Let us also assume that $f$ satisfies an Ambrosetti-Rabinowitz type condition. More specifically, we will assume that there exist $\theta >p$ and $M\in \mathbb{R}$ such that for all $s\in \mathbb{R}$, 
\begin{equation}\label{eq3}
	sf(s)  \geqslant \theta F(s)+M.   
\end{equation}
\begin{remark}
	The existence of at least one solution to our problem can be stated under the assumption $q\in (p-1 , \, p_s^*-1)$. The restriction $p-1 <q< \min \{ \frac{sp}{N}p_s^* , p_s^*-1 \} $ is necessary to prove the positiveness of this.     
\end{remark}
The aim of this paper is to prove the following result. 
\begin{theorem}[Main Theorem]
	Let us assume that $\Omega$ is a bounded domain with $C^{1,1}$ boundary. Then  there is $\lambda _0 >0$ such that for all $\lambda \in (0, \lambda_0)$ problem \eqref{eq1} has at least one positive weak solution $u_\lambda \in C^{\alpha}(\overline{\Omega})$, for some $\alpha \in (0,1).$
\end{theorem}
This result extends the one in \cite{CdFL} where the authors considered the problem for the $p$-Laplacian operator, ($2\leqslant p <N$). The difficulties to prove the positiveness of the solutions for Dirichlet problems with semipositone type nonlinearities are well documented, see for example \cite{BS}, \cite{CCSU} and references therein. Such issues persist in the nonlocal case.  To the best of our knowledge this is the first result on the existence of positive solutions for a semipositone nonlinearity with the fractional $p$-Laplacian.
In \cite{DT}, the authors studied the problem \eqref{eq1} with $p=2$, $f(u)=u^q-1$, (semipositone) but $0<q<1$. Indeed, they proved the existence of at least one positive solution if $\lambda >0$ is sufficiently large. In \cite{AHS}, the authors proved the existence of positive solutions of a problem of semipositone type for the $\Phi$-Laplacian through Orlicz-Sobolev spaces. \\
\noindent
Throughout this paper, $C$ will denote positive constant, not the same at each occurrence.

\section{Fractional frame}
\begin{definition} Let $s \in (0,1)$ and $1 \leq p < \infty$ and let
	\[
	W^{s,p}(\mathbb{R}^N):=\left\{u \in L^p(\mathbb{R}^N) : \int_{\mathbb{R}^{2N}}\frac{|u(x)-u(y)|^p}{|x-y|^{N+sp}}\D x \D y < \infty\right\}\]
	be the fractional Sobolev space endowed with the norm 
	\[\|u\|_{s,p}=(\|u\|_p^p+[u]_{s,p}^p)^{1/p},\]
	where \[[u]^p_{s,p}:=\int_{\mathbb{R}^{2N}}\frac{|u(x)-u(y)|^p}{|x-y|^{N+sp}}\D x\D y,\]
	is the Gagliardo seminorm and for every $1\leqslant q \leqslant \infty$, $\|\cdot \|_q$ is the norm in $L^q(\Omega)$.
\end{definition}
With this norm, $W^{s,p}(\mathbb{R}^N)$ is a Banach space.
We shall work in the closed subspace
\[W_0^{s,p}(\Omega):=\left\{u \in W^{s,p}(\mathbb{R}^N): u =0 ~\textrm{a.e in}~ \mathbb{R}^N - \Omega\right\}\]
which can be equivalently renormed by setting $\|u\|=[u]_{s,p}$. The equivalence of this norms is a consequence of the Sobolev embedding theorem (see \cite{DPV}).\\
Let us set for all $s\in \mathbb{R}$
\[\Phi_p(s)=|s|^{p-2}s. \]
A weak solution to the problem (\ref{eq1})
is a function $u \in W_0^{s,p}(\Omega)$ such that for all $\varphi \in W_0^{s,p}(\Omega)$
\begin{equation*}
	\int_{\mathbb{R}^{2N}}\frac{\Phi_p(u(x)-u(y))(\varphi (x)-\varphi (y))}{|x-y|^{N+sp}}\D x \D y =\lambda \int_{\Omega}f( u) \varphi  \D x    . 
\end{equation*}
We shall give to this problem a variational approach. Then, for each $\lambda>0$ let us define the functional $E_\lambda : W^{s,p}_0 (\Omega) \to \mathbb{R}$ as
\begin{equation}
	E_\lambda (u)=\frac{1}p \int _{\mathbb{R}^{2N}} \frac{|u(x)-u(y)|^p}{|x-y|^{N+sp}} \D x \, \D y- \lambda \int _{\Omega} F(u) \D x .
\end{equation}
Observe that $E_\lambda (u):=\frac{1}p\|u\|^p-\lambda \int _{\Omega} F(u) \D x$. 
It is well known that $E_\lambda \in C^1$ and its derivative is given by
\begin{equation}\label{deriv-E}
	\langle E'_{\lambda} u, \, \varphi  \rangle    =\int _{\mathbb{R}^{2N} }   \frac{\Phi_p(u(x)-u(y))(\varphi (x)-\varphi  (y))}{|x-y|^{N+sp}} \D x \, \D y - \lambda \int _{\Omega} f(u) \varphi  \D x .
\end{equation}
Therefore, the critical points of $E_\lambda$ turns out to be the weak solutions of problem \eqref{eq1}.

\section{Preliminary results}
\noindent 
In this section we shall establish some lemmas that guarantee that $E_\lambda$ has a critical point, $u_\lambda$, whenever $\lambda>0$ is sufficiently small. After that, we present some lemmas concerning the regularity of $u_\lambda$. Finally we prove our main result.
The positive number 
\[r:=\frac{1}{q+1-p} ,\]
will be use repeatedly throughout this paper. Let $\varphi \in W_0^{s,p}(\Omega)$ be a positive function with $\| \varphi \|=1$ and let  
\[c:=\bigg(\frac{2}{pA_1\|\varphi\|_{q+1}^{q+1}}\bigg)^r>0.\]
Finally, let us define $\D _\Omega (x):=\dist (x, \Omega^c)$, for all $x\in \mathbb{R}^N$.
\begin{lemma}\label{lema1}
	There exists $\lambda_1 >0$ such that if $\lambda \in (0, \lambda_1)$ then $E_\lambda (c\lambda ^{-r}\varphi) \leqslant 0 $.
\end{lemma}
\begin{proof} Let $l=c \lambda^{-r}$. From the growth behaviour of $F$ (see \eqref{lower_bound_F}) and the fact that $\|\varphi\|=1$  we have 
	\begin{align}\label{E<0}
		\begin{split}
			E_{\lambda}(l \varphi) & =\frac{1}{p}\|l \varphi\|^p- \lambda \int_{\Omega} F(l \varphi) \D x\\
			& \leqslant \frac{l^p}{p}\| \varphi\|^p- \lambda A_1 l^{q+1} \int_{\Omega} \varphi^{q+1} \D x+ \lambda A_1C_1 |\Omega|\\
			& \leqslant \frac{l^p}{p}- \lambda A_1 l^{q+1}\|\varphi\|^{q+1}_{q+1}+\lambda A_1C_1 |\Omega| .
		\end{split}
	\end{align}
	Thus, if $0<\lambda < \big(\frac{c^p}{2pA_1C_1|\Omega|}\big)^{1/(1+rp)}=:\lambda_1$, then 
	\begin{equation}
		E_{\lambda}(l \varphi) \leqslant -\frac{c^p}{2p}\lambda^{-rp} \leqslant 0 . 
	\end{equation}
\end{proof}

\begin{lemma}\label{lema2}
	There exist $\tau>0$, $c_1>0$ and  $0<\lambda_2<1 $ such that if $\|u\|=\tau \lambda ^{-r}$ then $E_\lambda (u) \geq c_1(\tau \lambda^{-r})^p$ for all $\lambda \in (0, \lambda_2)$.
\end{lemma}
\begin{proof}
	Let $u \in W_0^{s,p}(\Omega)$ with $\|u\|=\lambda^{-r}\tau$, by the Sobolev embedding theorem, there exists $K_1 >0$ such that for all $v\in W_0^{s,p}(\Omega)$, $\|v\|_{q+1}\leq K_1 \|v\|$, define $\tau = \min  \{(2pK_1^{q+1}B_1)^{-r}, c\}$ then,
	\begin{align*}
		E_{\lambda}(u) &= \frac{1}{p}\|u\|^p- \lambda \int_{\Omega}F(u) \D x\\
		& \geq \frac{1}{p}(\lambda^{-r}\tau)^p- \lambda B_1 \|u\|_{q+1}^{q+1}- \lambda B_1|\Omega|\\
		& \geq \frac{1}{p}(\lambda^{-r}\tau)^p-\lambda B_1(K_1\|u\|)^{q+1}- \lambda B_1 |\Omega|\\
		& =  \frac{1}{p}(\lambda^{-r}\tau)^p-\lambda B_1K_1^{q+1}(\lambda^{-r}\tau)^{q+1}- \lambda B_1 |\Omega|\\
		& \geq \lambda^{-rp}\left(\frac{\tau^p}{2p}-\lambda^{1+rp}|\Omega|B_1 \right)\\
		& \geq \lambda^{-rp}\frac{\tau^p}{4p}
	\end{align*}
	taking $c_1=\frac{1}{4p}$ and $\lambda_2:=\tau^{p/(1+rp)}(4pB_1|\Omega|)^{-1/(1+rp)}$ we obtain the result. 
\end{proof}

\begin{lemma}\label{critical-point}
	Let $\lambda_3=\min \{ \lambda_1 , \lambda_2 \}$. Then, there exists a constant $c_2 >0$ such that for all $\lambda \in (0, \lambda_3) $ the functional $E_\lambda$ has a critical point $u_\lambda$ which satisfies 
	\begin{equation*}
		c_1 \lambda^{-rp} \leqslant E_\lambda(u_\lambda) \leqslant c_2 \lambda^{-rp} ,
	\end{equation*}
	where $c_1>0$ is the constant given in Lemma \ref{lema2}.
\end{lemma}
\begin{proof}
	First of all, we will prove that $E_\lambda$ satisfies the Palais-Smale condition. Let us assume that $\{u_n\}$ is a sequence in $W_0^{s,p}(\Omega)$ such that $\{E_{\lambda}(u_n) \}$ is bounded and $E'_\lambda (u_n) \to 0,$ as $n \to \infty .$ Hence, there exists $\nu >0$ such that for all $n > \nu$
	
	\begin{equation*}
		| \langle E_{\lambda}'(u_n),u_n \rangle| \leqslant \|u_n\|.
	\end{equation*}
	Moreover, from \eqref{deriv-E} we have
	\begin{equation}\label{eq10}
		-\|u_n\|^p-\|u_n\|\leqslant -\lambda \int_{\Omega} f(u_n)u_n \D x , \quad \text{for all} \, n > \nu .
	\end{equation}
	Let $K>0$ such that for all $n$, $|E_\lambda (u_n)|\leqslant K$. From the Ambrosetti-Rabinowitz condition (equation \eqref{eq3}) we see that
	\begin{align}\label{eq11}
		\begin{split}
			\frac{1}{p}\|u_n\|^{p}-\frac{\lambda}{\theta}\int_{\Omega}f(u_n)u_n \D x+\frac{\lambda}{\theta}M|\Omega|
			& \leqslant \frac{1}{p}\|u_n\|^{p}-\lambda\int_{\Omega}F(u_n)\D x \\
			&  \leqslant K.
		\end{split}
	\end{align}
	
	Using (\ref{eq10}) and (\ref{eq11}) we obtain
	\[
	\left(\frac{1}{p}-\frac{1}{\theta} \right) \|u_n\|^p-\frac{1}{\theta}\|u_n\| 
	\leqslant K-\frac{\lambda}{\theta}M |\Omega|,
	\]
	which proves that $\{u_n\}$ is bounded in $W_0^{s,p}(\Omega)$. Therefore, up to a sub-sequence, $\{u_{n}\}$ converges weakly to the function $u\in W_0^{s,p}(\Omega)$. Since $p<q+1<p_s^*$, then $u_n \to u$ (strongly) in $L^{q+1}(\Omega)$. Applying the Hölder inequality this implies that
	\begin{equation*}\label{eq12}
		\lim_{n\to \infty} \lambda\int_{\Omega}f(u_n)(u_n-u)\D x = 0.
	\end{equation*}
	Then, since $\lim_{n \to \infty}E'_{\lambda}(u_n)=0$, we have
	\begin{equation}\label{eq13}
		\lim_{n\to \infty} \int_{\mathbb{R}^{2N}} \frac{\Phi_p(u_n(x)-u_n(y))((u_n-u)(x)-(u_n-u)(y))}{|x-y|^{N+sp}}=0.
	\end{equation}
	Using again that $u$ is the weak limit of $u_n$ we have
	\begin{equation}\label{eq14}
		\lim_{n \to \infty}\int_{\mathbb{R}^{2N}} \frac{\Phi_p(u(x)-u(y))((u_n-u)(x)-(u_n-u)(y))}{|x-y|^{N+sp}}=0.
	\end{equation}
	On the other hand, taking into account the Hölder inequality, we see that
	\begin{align*}
		& \int_{\Omega} \frac{\Phi_p(u_n(x)-u_n(y))-\Phi_p(u(x)-u(y))}{|x-y|^{N+sp}}   ((u_n-u)(x)-(u_n-u)(y)) \D x \D y  & \\
		&    = \int_{\Omega} \left[\frac{|u_n(x)-u_n(y)|^{p}}{|x-y|^{N+sp}}- \frac{\Phi_p(u_n(x)-u_n(y))(u(x)-u(y))}{|x-y|^{N+sp}} \right.\\
		&   \left.-\frac{\Phi_p(u(x)-u(y))(u_n(x)-u_n(y))}{|x-y|^{N+sp}}+ \frac{|u(x)-u(y)|^p}{|x-y|^{N+sp}} \right] \D x \D y\\
		& \geqslant \|u_n\|^p-\|u_n\|^{p-1}\|u\|-\|u_n\|\|u\|^{p-1}+\|u\|^p\\
		& = ([\|u_n\|^{p-1}-\|u\|^{p-1})(\|u_n\|-\|u\|) \geqslant 0 .
	\end{align*}
	From (\ref{eq13}), (\ref{eq14})  we obtain  
	\begin{equation*}
		\lim_{n \to \infty} (\|u_n\|^{p-1}-\|u\|^{p-1})(\|u_n\|-\|u\|)=0 ,
	\end{equation*}
	which implies
	\begin{equation*}
		\lim_{n \to \infty}\|u_n\|=\|u\| .
	\end{equation*}
	Since  $u_n \rightharpoonup u$, then $u_n \to u$ strongly in $W_0^{s,p}(\Omega)$. This proves that $E_{\lambda}$ satisfies the Palais-Smale condition.\\
	Let us observe that, from \eqref{E<0}, for all $0\leqslant l \leqslant c\lambda ^{-r}$
	\[E_\lambda (l\phi ) \leqslant \frac{l^p}p+\lambda A_1C_1 |\Omega| \leqslant \frac{c^p}p \lambda^{-rp}+ A_1C_1 |\Omega|\lambda^{-rp} =c_2\lambda^{-rp} . \]
	where $c_2:=\frac{c^p}p + A_1C_1 |\Omega|$. 
	Therefore
	\begin{equation}\label{eq17}
		\max_{0\leqslant l \leqslant c\lambda ^{-r}}  E_{\lambda}(l \phi)   \leqslant c_2 \lambda^{-rp}.
	\end{equation}
	From Lemmas \ref{lema1} and \ref{lema2},  and the Mountain Pass Theorem for each $\lambda \in (0, \lambda_3 ) $ there exist $u_\lambda \in W_0^{s,p}(\Omega)$ such that $E'_\lambda (u_\lambda)=0$.
	Furthermore, this critical point is characterized by
	\begin{equation}\label{E-min-max}
		E_{\lambda}(u_{\lambda})= \min_{\gamma \in \Gamma} \max_{0 \leqslant t \leqslant 1} E(\gamma(t)).
	\end{equation}
	where $\Gamma$ is the set of continuous functions $\gamma:[0,1] \to W_0^{s,p}(\Omega)$ with $\gamma(0)=0$, $\gamma(1)=c\lambda ^{-r}\varphi$. Moreover, from \eqref{eq17}, \eqref{E-min-max} and Lemma \ref{lema2} we see that
	\begin{equation*}
		c_1\tau^p \lambda^{-rp} \leqslant E_\lambda(u_\lambda) \leqslant c_2 \lambda^{-rp} .
	\end{equation*}
	Note that $c_1$ $c_2 $ are independent of $\lambda$.
\end{proof}

\begin{remark}\label{upper-bound-||u||}
	There exists a constant $C>0$ such that for all $0<\lambda <\lambda_3$
	\begin{equation}
		\|u_\lambda\| \leqslant C\lambda^{-r}.
	\end{equation}
	In fact, since $u_\lambda$ is a critical point of $E_\lambda$, then
	\[\|u_\lambda \|^p=\lambda \int_\Omega f(u_\lambda)u_\lambda \D x .\]
	From the Ambrosetti-Rabinowitz condition  and Lemma \ref{critical-point} we see that
	\begin{eqnarray*}
		\bigg( \frac{1}p-\frac{1}{\theta} \bigg)\|u_\lambda \|^p &\leqslant& \frac{1}p\|u_\lambda \|^p -\frac{\lambda}{\theta} \int_\Omega f(u_\lambda) u_\lambda \D x +\frac{\lambda}{\theta} M|\Omega| \\
		&\leqslant& \frac{1}p\|u_\lambda \|^p -\lambda \int_\Omega F(u_\lambda)  \D x \\
		&=&E_\lambda (u_\lambda) \\ &\leqslant& c_2 \lambda^{-rp}.
	\end{eqnarray*}
\end{remark}

\begin{lemma}
	There exist $\alpha \in (0, s]$ and a constant $C>0$ such that for all $0<\lambda <\lambda_3$, the solution  $u_\lambda$  of the problem \eqref{eq1} satisfies $u_\lambda /\D _\Omega ^s \in C^\alpha (\overline{\Omega)}$ and 
	\[ \bigg\| \frac{u_\lambda}{\D ^s_\Omega}\bigg\|_{C^{\alpha}(\overline{\Omega})} \leqslant C \lambda^{-r} . \]    
\end{lemma}
\begin{proof}
	Let $t$ be such that $\frac{N}{sp}<t$ and $tq<p_s^*$ and $g:=\lambda f\circ u_\lambda .$ Since $W_0^{s,p}(\Omega) \subseteq L^{tq}(\Omega)$ and  $|g|\leqslant A_1 \lambda (|u_\lambda|^{q}+1)$, then $g\in L^{t}(\Omega)$. According to Lemma 2.3 from \cite{MPSY}, 
	\begin{equation}\label{eq1-lemma4}
		\|u_\lambda\|_\infty \leqslant \|g\|_t^{\frac{1}{p-1}} .\end{equation} 
	But taking into account the Remark \ref{upper-bound-||u||}, we have
	\[  \|g\|_t \leqslant C\lambda \|u_\lambda\|_{tq}^q \leqslant C\lambda \|u_\lambda\|^q \leqslant    C\lambda^{1-rq}.\]
	Therefore, from \eqref{eq1-lemma4} and $-r=(1-rq)/(p-1)$, we see that
	\begin{equation}\label{upper-bound-u-infty}
		\|u_\lambda\|_\infty \leqslant C \lambda^{-r}.
	\end{equation}
	Since $u_\lambda \in L^{\infty}(\Omega)$ then $g \in L^\infty(\Omega)$. From Theorem 1.1. in \cite{IMS1}, we see that there exists $\alpha \in (0, s]$ and $C>0$, depending only on $N,p,s$ and $\Omega$, such that the solution $u_\lambda $ satisfies $u_\lambda /\D _\Omega ^s \in C^\alpha (\overline{\Omega})$ and 
	\[\bigg\| \frac{u_\lambda}{\D ^s_\Omega}\bigg\|_{C^{\alpha}(\overline{\Omega})} \leqslant C\| \lambda f(u_\lambda ) \|^\frac{1}{p-1}_{\infty} \leqslant \lambda^{-r},\]
	where the last inequality was obtained taking into account \eqref{upper-bound-u-infty}, the growing condition of $f$ and that $1-rq=-r(p-1)$. 
\end{proof}

\begin{lemma}\label{lema3}
	Let  $  u_\lambda$ be a weak solution of \eqref{eq1}. Then there exists a constant $C$ such that for all $0<\lambda<\lambda_3 $
	\[\ C \lambda^{-r} \leqslant \| u_\lambda\|_\infty .\]
\end{lemma}
\begin{proof}
	From Lemma \ref{critical-point} there exists $c_1$ such that $c_1\lambda^{-rp} \leqslant E_\lambda (u_\lambda)$. Moreover, since $\min F >- \infty$ then 
	\begin{align}\label{est-inf1-u}
		\begin{split}
			\lambda \int_\Omega f(u_\lambda) u_\lambda\D x &= \|u_\lambda\|^p \\
			&=pE_\lambda (u_\lambda)+p\lambda \int_\Omega F(u_\lambda)\D x \\
			&\geqslant pc_1 \lambda ^{-rp} +p |\Omega| \lambda \min F \\
			&\geqslant C_1 \lambda ^{-rp},
		\end{split}
	\end{align}
	for some $C_1>0$. On the other hand, observe from \eqref{eqn4} that there exists $B_2>0$ such that for all $s\in \mathbb{R}$, $f(s)s\leqslant B_2(|s|^{q+1}+|s|)$. Thus
	\begin{align}\label{est-inf2-u}
		\begin{split}
			\lambda \int_\Omega f(u_\lambda) u_\lambda\D x &\leqslant  B_2 \lambda \int_\Omega (|u_\lambda|^{q+1}+|u_\lambda|)\D x \\
			&\leqslant B_2\lambda \int_\Omega (\|u_\lambda\|_\infty^{q+1}+\|u_\lambda\|_\infty)\D x  \\
			&\leqslant  B \lambda \|u_\lambda\|_\infty^{q+1},
		\end{split}
	\end{align}
	for some $B>0$. From \eqref{est-inf1-u} and \eqref{est-inf2-u} we obtain the result. 
\end{proof}
Finally we prove the Main Theorem.  

\begin{proof}[Proof of the Main Theorem]
	Arguing by contradiction, let $\{\lambda_{j}\}$ a sequence of positive numbers such that $\lambda _j \to 0$, as $j \to  \infty $ and such that $|\{ x \in \Omega : \ \ u_{\lambda_ j}(x) \leq 0 \}|>0.$
	Let $ w_j:=\frac{u_{\lambda _j}}{\|u_{\lambda _j}\|_\infty}.$
	Then
	\[(-\Delta )_p^s(w_j)=\lambda _j f(u_{\lambda_j}) \|u_{\lambda_j}\|_\infty ^{1-p}.\]
	By Lemma \ref{lema3} and Theorem 1.1 of \cite{IMS1}, there exists $\alpha \in (0,s]$ such that
	\[\bigg\|\frac{w_j}{\D^s_\Omega }\bigg\|_{C^\alpha(\overline{\Omega})} \leqslant \|\lambda _j f(u_{\lambda_j}) \|u_{\lambda_j}\|_\infty ^{1-p}\|_\infty ^{\frac{1}{p-1}} \leq C,\]
	where $C$ does not dependent on $\lambda_j$. Let us choose any $0< \beta <\alpha$.  Since $C^{\alpha} (\overline{\Omega})\subset \subset C^{\beta} (\overline{\Omega})$ (see Theorem 5.14, \cite{A}) then, up to a sub-sequence, $\lim_{j\to\infty}  \frac{w_j}{\D^s_\Omega }=\frac{w}{\D^s_\Omega }$ in $C^\beta(\overline{ \Omega})$.  
	Now, we will use comparison principle to prove that $w(x)\geqslant 0$.
	Let $v_0\in W_0^{s,p}(\Omega)$ be the solution of 
	\[	\displaystyle \left\{\begin{array}{rcll}
		(-\Delta)_p^s u&=&1, \qquad &\text{in $\Omega$}\\
		u&=&0,  &\text{in $\mathbb{R}^N-\Omega$}.
	\end{array}\right.\]
	Let $K_j=\frac{\lambda_j}{\|u_{\lambda_j}\|^{p-1}_{\infty}}\min_{t\in \mathbb{R}} f(t).$  Observe that $K_j<0$.
	Then, the solution $v_j\in W_0^{s,p}(\Omega)$ of 
	\[\displaystyle \left\{\begin{array}{rcll}
		(-\Delta)_p^s u&=&K_j, \qquad &\text{in $\Omega$}\\
		u&=&0,  &\text{in $\mathbb{R}^N-\Omega$},
	\end{array}\right.\]
	is given by $v_j=-(-K_j)^{1/(p-1)}v_0$.
	Since $\lambda _j f(u_{\lambda_j}) \|u_{\lambda_j}\|_\infty ^{1-p} \geqslant K_j$. By the comparison principle stated in \cite{IMS} (Proposition 2.10) $w_j\geqslant v_j$. 
	Since $v_j\to 0$, as $j\to \infty$, then  $w (x)\geqslant 0$.\\
	Let us observe that since  $\{\lambda _j f(u_{\lambda_j}) \|u_{\lambda _j}\|_\infty ^{1-p}\}_j$
	is bounded by a constant independent of $\lambda_j$, then there exists $t>1$ such that  $\{\lambda _j f(u_{\lambda_j}) \|u_{\lambda_j}\|_\infty ^{1-p}\}_j$ is bounded in $L^t(\Omega)$. Thus, we may assume  that it converges weakly in $L^t(\Omega)$. Let $z:=\lim_{j\rightharpoonup 0 } \lambda _j f(u_{\lambda_j}) \|u_{\lambda_j}\|_\infty ^{1-p}$, its weak limit.
	Since $f$ is bounded from below and  $\lim_{j \to \infty} \lambda _j \|u_{\lambda_j}\|_\infty ^{1-p} =0,$ then $z \geqslant 0.$ We claim that $(-\Delta)_p^s(w)=z$. In fact, from remark \ref{upper-bound-||u||} and Lemma \ref{lema3}, the sequence of functions 
	\[\psi_j(x,y):=\frac{|w_j(x)-w_j(y)|}{|x-y|^{\frac{N}{p}+s}},\] is bounded in $L^p(\mathbb{R}^{2N})$. Therefore, following the same procedure made in Lemma \ref{critical-point} to prove the strong convergence of $\{u_n\}$ (see Lemma \ref{app-Lema_2} in the appendix), we conclude that it converges to 
	\[\psi(x,y):=\frac{|w(x)-w(y)|}{|x-y|^{\frac{N}{p}+s}},\] in $L^{p}(\mathbb{R}^{2N}).$ Then there exists $h\in L^{p}(\mathbb{R}^{2N})$ such that $|\psi_j (x,y)| \leqslant h(x,y)$, a.e. $(x,y)$. Hence, from the Young's inequality, for all $\varphi \in W^{s,p}_0(\Omega)$ we have 
	\begin{eqnarray*}
		\frac{|w_j(x)-w_j(y)|^{p-1}|\varphi(x)-\varphi(y)|}{|x-y|^{N+sp}} &= & \frac{|w_j(x)-w_j(y)|^{p-1}|\varphi(x)-\varphi(y)|}{|x-y|^{\frac{N+sp}{p'}}|x-y|^\frac{N+sp}{p}} \\
		& \leqslant & \frac{1}{p'} \frac{|w_j(x)-w_j(y)|^{(p-1)p'}}{|x-y|^{N+sp}} + \frac{1}{p}\frac{|\varphi(x)-\varphi(y)|^p}{|x-y|^{N+sp}}\\
		&\leqslant & \frac{1}{p'} (h(x,y))^p + \frac{1}{p}\frac{|\varphi(x)-\varphi(y)|^p}{|x-y|^{N+sp}},
	\end{eqnarray*}
	where $p'$ stands for the conjugate Hölder exponent of $p$.
	Since the last function belongs to $L^1(\mathbb{R}^{2N})$, by the Lebesgue Dominated Convergence Theorem we have
	\begin{align}\label{w-weak-sol}
		\begin{split}
			& \int_{\mathbb{R}^{2N}  }  \frac{|w(x)-w(y)|^{p-2}(w(x)-w(y))(\varphi (x)-\varphi (y))}{|x-y|^{N+sp}}  \D x \, \D y
			\\
			& = \lim_{j \to \infty} \int_{\mathbb{R}^{2N} }  \frac{|w_j(x)-w_j(y)|^{p-2}(w_j(x)-w_j(y))(\varphi (x)-\varphi (y))}{|x-y|^{N+sp}}  \D x \, \D y \\
			&  = \lim_{j \to \infty} \int_{\Omega }  \lambda _j f(u_{\lambda_j}(x)) \|u_{\lambda_j}\|_\infty ^{1-p} \varphi (x) \D x \\
			& = \int_{\Omega }  z(x) \varphi (x) \D x .
		\end{split}
	\end{align}
	Observe that we also proved that $w_j\to w$ in $W_0^{s,p} (\Omega)$, and thus $w\in W_0^{s,p} (\Omega)$. This proves the claim. Thus $w$ is  a supersolution of the $(-\Delta)_p^s(w)=0$ in $\Omega$. Since $\Omega$ has $C^{1,1}$ boundary then it satisfies the interior ball condition (see Theorem 1.0.9 in \cite{T}). Therefore, by Theorems 1.4 and 1.5 of \cite{PQ} we have $w>0$ in $\Omega$ and  for all $x_0 \in \partial \Omega$, 
	\[
	\liminf_{ x\to x_0} \frac{w(x)}{\D^s_{B_R(x)}}>0,
	\]
	where $B_R\subseteq \Omega$ and $x_0\in \partial B_R $.
	From Lemma \ref{lemma-hopf2} (see appendix), there exists $j$ sufficiently large such that $w_j >0$ in $\Omega .$ Absurd. 
\end{proof}

\section{Appendix}

In this section we shall prove some technical results. The first one is based on the Hopf's Lemma established in \cite{PQ}. The second, follows the same lines in part of the proof of Lemma \ref{critical-point}.
\begin{lemma}\label{lemma-hopf2}
	Let us assume that $\Omega\subseteq \mathbb{R}^N$ is bounded domain with $C^{1,1}$ boundary and $\frac{w_j}{\D _\Omega ^s} \to \frac{w}{\D _\Omega ^s}$ in $C^\beta\big(\overline{\Omega}\big)$ with $w(x)=w_j(x)=0$, for all $j$ and all $x\in \partial\Omega$. Let us assume that $w>0$ in $\Omega$ and for all $x_0\in \partial \Omega$ 
	\begin{equation}\label{hopf1}
		m:=\liminf_{ x\to x_0} \frac{w(x)}{\D^s_{B_R}(x)}>0 . 
	\end{equation}
	Then there exists 
	$j$ such that $w_j(x)>0$ for all $x\in \Omega$.
\end{lemma}
\begin{proof}
	First of all, let us emphasize that, since $\frac{w}{\D _\Omega ^s} \in C^\beta\big(\overline{\Omega}\big)$, then for all $x_0\in \partial\Omega$, $\frac{w(x_0)}{\D ^s_\Omega (x_0)}$ is well defined in terms of limits.  
	Now, let  $B_R \subseteq \Omega$ be an interior ball such that $x_0\in \partial B_R$ and let be $\varepsilon_0 >0$ such that for all $x\in B_R\cap B(x_0, \varepsilon_0)$, 
	\[\frac{w(x)}{\D ^s_{B_R} (x)}> \frac{m}2.\]
	Let us pick up a sequence $\{x_n\} $ in $B_R\cap B(x_0, \varepsilon_0)$ in the segment joining $x_0$ and the center of $B_R$ and such that $x_n \to x_0$. So that for all $n$, $x_n-x_0$ is orthogonal to $\partial B_R$ and $\partial \Omega$ and $\D_{B_R}(x_n)=\D_{\Omega}(x_n)$. Therefore
	\[\frac{w(x_0)}{\D^s_\Omega(x_0)}=\lim_{ n\to\infty} \frac{w(x_n)}{\D^s_\Omega(x_n)}=\lim_{ n\to\infty} \frac{w(x_n)}{\D^s_{B_R}(x_n)}\geqslant \frac{m}2>0 .\]
	And, obviously, $\frac{w(x)}{\D^s_\Omega(x)}>0$
	for all $x\in \Omega$. Thus $\frac{w}{\D^s_\Omega}$ is positive in the compact 
	$\overline{\Omega}.$  Let \begin{equation}\label{Hopf2}
		\varepsilon:=\min \frac{w}{\D^s_\Omega} >0.
	\end{equation} 
	Let $\Omega_1$ be a nonempty open set such that $\overline{\Omega}_1\subseteq\Omega$.  
	We claim that there exists $j$ such that for all $x\in \overline{\Omega_1}$, $w_j(x)>0$. Indeed, there exists $j$ sufficiently large such that 
	\[\bigg\|\frac{w}{\D _\Omega ^s}-\frac{w_j}{\D _\Omega ^s}\bigg\|_{C^\beta(\overline{\Omega})} <\frac{\epsilon}{2} .\]
	In particular for all $x\in \overline{\Omega}_1$
	\[-\frac{\epsilon}{2}\leqslant \frac{w_j(x)}{\D _\Omega ^s(x)}-\frac{w(x)}{\D _\Omega ^s(x)} . \]
	Then, for all $x\in \overline{\Omega}_1$
	\[\frac{\epsilon}{2}\leqslant \frac{w(x)}{\D _\Omega ^s(x)}-\frac{\epsilon}{2} <\frac{w_j(x)}{\D _\Omega ^s(x)}.\]
	Which proves the claim. Finally, we will prove that for all $x\in \Omega-\overline{\Omega_1}$, $w_j(x)>0$. Let us argue by contradiction. If there exists $x_0\in \Omega-\overline{\Omega_1}$ such that $w_j(x_0)\leq 0$, then, by the intermediate Value Theorem, there is $z_0\in \Omega-\overline{\Omega_1}$ such that $w_j(z_0)=0$. Thus, from \eqref{Hopf2} and the definition of $\varepsilon_1$, we have 
	\begin{eqnarray*}
		\epsilon  & \leqslant & \bigg| \frac{w(z_0)}{\D _\Omega ^s(z_0)}-\frac{w_j(z_0)}{\D _\Omega ^s(z_0)}\bigg| \\
		&\leqslant & \bigg\|\frac{w}{\D _\Omega ^s}-\frac{w_j}{\D _\Omega ^s}\bigg\|_{C^\beta(\overline{\Omega})} <\frac{\epsilon}{2} .
	\end{eqnarray*}
	Absurd.
\end{proof}

\begin{lemma}\label{app-Lema_2}
	Let $\{w_j\}$ be a bounded sequence in $W_0^{s,p}(\Omega)$, such that
	\[\displaystyle \left\{\begin{array}{rcll}
		(-\Delta)_p^s (w_j)&=&\lambda_j g(w_j) ~\textrm{in} \qquad& \Omega\\
		w_j(x)&=&0~~~~~~~~~\textrm{in}& \mathbb{R}^N-\Omega ,
	\end{array}\right.\]	
	with  $\{\lambda_jg(w_j)\}$ bounded in $L^\infty(\Omega)$.
	Then $w_j$ converges strongly in $W_0^{s,p}(\Omega)$.
\end{lemma}
\begin{proof}
	Since $\{w_j\}$ is bounded in $W_0^{s,p}(\Omega)$, then, up to a sub-sequence, $\{w_j\}$ converges weakly to the function $v\in W_0^{s,p}(\Omega)$. Since $p<q+1<p_s^*$, then $w_j \to v$ (strongly) in $L^{q+1}(\Omega)$. As $\{\lambda_jg(w_j)\}$ bounded in $L^\infty(\Omega)$, applying the Hölder inequality this implies that
	\begin{equation*}\label{append-eq12}
		\lim_{j\to \infty} \lambda_j\int_{\Omega}g(w_j)(w_j-v)\D x = 0.
	\end{equation*}
	Then, since $J'_{\lambda_j}(w_j)=0$ (where $J_\lambda$ is the associated Energy Functional to this problem), we have
	\begin{equation}\label{append-eq13}
		\lim_{j\to \infty} \int_{\mathbb{R}^{2N}} \frac{\Phi_p(w_j(x)-w_j(y))((w_j-v)(x)-(w_j-v)(y))}{|x-y|^{N+sp}}=0.
	\end{equation}
	Using again that $v$ is the weak limit of $w_j$ we have
	\begin{equation}\label{append-eq14}
		\lim_{j \to \infty}\int_{\mathbb{R}^{2N}} \frac{\Phi_p(v(x)-v(y))((w_j-v)(x)-(w_j-v)(y))}{|x-y|^{N+sp}}=0.
	\end{equation}
	Thus, from the same argument that we use in the proof of Lemma \ref{critical-point} we obtain 
	\begin{align*}
		& \int_{\Omega} \frac{\Phi_p(w_j(x)-w_j(y))-\Phi_p(v(x)-v(y))}{|x-y|^{N+sp}}   ((w_j-v)(x)-(w_j-v)(y)) \D x \D y  & \\
		&     \geqslant (\|w_j\|^{p-1}-\|v\|^{p-1})(\|w_j\|-\|v\|) \geqslant 0 .
	\end{align*}
	From (\ref{append-eq13}), (\ref{append-eq14})  we obtain  
	\begin{equation*}
		\lim_{j \to \infty} (\|w_j\|^{p-1}-\|v\|^{p-1})(\|w_j\|-\|v\|)=0 ,
	\end{equation*}
	which implies
	\begin{equation*}
		\lim_{j \to \infty}\|w_j\|=\|v\| .
	\end{equation*}
	Since  $w_j \rightharpoonup v$, then $w_j \to v$ strongly in $W_0^{s,p}(\Omega)$.
\end{proof}

\section*{Acknowledgment}

E. Lopera was partially supported by Facultad de Ciencias, Universidad Nacional de Colombia, sede Manizales, Hermes codes 55156 and 51894, and 100,000 Strong in the Americas, Innovation Fund.\\
C. López was partially supported by Facultad de Ciencias, Universidad Nacional de Colombia, sede Manizales, Hermes code 55156 and 100,000 Strong in the Americas, Innovation Fund.\\
R. E. Vidal was partially supported by CONICET, FONCyT and SECyT-UNC.

\end{document}